\documentclass[12pt,a4paper]{article}

\usepackage[utf8]{inputenc}
\usepackage[OT1]{fontenc}
\usepackage{amsmath}
\usepackage{amsthm}
\usepackage{amsfonts}
\usepackage{amssymb}
\usepackage{graphicx}
\usepackage{cite}
\usepackage{hyperref}
\usepackage{xcolor}
\usepackage{tikz}

\newcommand{\ii}{\mathrm{i}}

\usepackage[left=2cm,right=2cm,top=2cm,bottom=2cm]{geometry}
\graphicspath{{pic/}}

\newtheorem{proposition}{Proposition}[section]
\newtheorem{lemma}[proposition]{Lemma}%[section]
\newtheorem{theorem}[proposition]{Theorem}%[section]
\newtheorem{corollary}[proposition]{Corollary}%[section]
\newtheorem{conjecture}{Conjecture}

\theoremstyle{definition}
\newtheorem{definition}[proposition]{Definition}%[section]
\newtheorem{remark}[proposition]{Remark}
\newtheorem*{akn}{Acknowledgements}

\newtheorem*{solution*}{Solution}
\newtheorem{ass}{Assumption}

\DeclareMathOperator{\supp}{supp}

\DeclareMathOperator{\dive}{div}

 \newcommand{\bbar}[1]{\setbox0=\hbox{$#1$}\dimen0=.2\ht0 \kern\dimen0 \overline{\kern-\dimen0 #1}}

 \DeclareMathOperator{\End}{\ensuremath{\mathcal{E}\kern-.125em\mathpzc{nd}}}

 \DeclareMathOperator{\Hom}{\mathcal{H}\kern-.125em\mathpzc{om}}
 
 \DeclareMathOperator{\Proj}{\mathcal{P}\kern-.125em\mathpzc{roj}}

 \renewcommand{\setminus}{\smallsetminus}

 \newcommand{\udot}{\ensuremath{{\lower .183333em \hbox{\LARGE \kern -.05em$\cdot$}}}}

 \DeclareMathOperator{\Vect}	{Vec}

 \newcommand{\cA}{\mathcal{A}}

\newcommand{\cS}{\mathcal{S}}
 
 \newcommand{\cU}{\mathcal{U}}

\newcommand{\C}{\mathbb{C}}

 \newcommand{\N}{\mathbb{N}}

 \newcommand{\Q}{\mathbb{Q}}
 \newcommand{\R}{\mathbb{R}}

 \newcommand{\p}{\partial}

 \DeclareMathOperator{\Vol}{Vol}

\author{Ivan Beschastnyi\footnote{i.beschastnyi@gmail.com, 
Inria, Sorbonne Universit\'e, Universit\'e de Paris, CNRS, Laboratoire Jacques-Louis Lions,  Paris, France}\,, 
Ugo Boscain\footnote{ugo.boscain@upmc.fr, CNRS, Sorbonne Universit\'e, Universit\'e de Paris, Inria, Laboratoire Jacques-Louis Lions,  Paris, France}%\;$^\S$
\,, Mario Sigalotti \footnote{mario.sigalotti@inria.fr, Inria, Sorbonne Universit\'e, Universit\'e de Paris, CNRS, Laboratoire Jacques-Louis Lions,  Paris, France}}

\title{An
obstruction to small-time controllability of the bilinear Schr\"odinger equation\footnote{This work was supported by the ANR project SRGI ANR-15- CE40-0018 and by the ANR project Quaco ANR-17-CE40-0007-01.}}

\begin{document}

\maketitle

\begin{abstract}
In this article we discuss which controllability properties of classical Hamiltonian systems are preserved after quantization. We discuss some necessary and some sufficient conditions for small-time controllability of classical systems and quantum systems using the WKB method. In particular, we investigate the conjecture that if the classical system is not small-time controllable, then the corresponding quantum system is not small-time controllable either. 
\end{abstract}

\section{Introduction}
\label{sec:intro}

The relation between classical and quantum dynamics has been an important question for physicists and mathematicians since the very beginning of quantum mechanics. The set of tools which allow to study a quantum system using the properties of the corresponding classical system is known as semi-classical analysis. Such methods try to capture the general idea that, when the Planck constant goes to zero, quantum dynamics approaches classical dynamics. In this way many interesting properties of a classical system can be observed in the behaviour of the corresponding quantum system (see for instance~\cite{zworski,nonnen} and the references therein).

The goal of this article is to explore which properties of the controlled Schr\"odinger equation
\begin{equation}
\label{eq:schro}
\ii\p_t \psi = -\frac{\Delta \psi}{2} + V(x)\psi + u(t) W(x) \psi, \quad \psi \in L^2(M),
\end{equation}
on a complete Riemannian manifold $M$ can be deduced from the behaviour of the corresponding controlled classical system
\begin{equation}
\label{eq:class}
\dot{\lambda} = \vec{H}(t,\lambda), \quad \lambda \in T^*M,
 \end{equation} 
with Hamiltonian
\begin{equation}
\label{eq:ham_class}
H(t,\lambda) = \frac{\|p\|^2}{2} + V(x)+ u(t) W(x),
\end{equation}
where $\lambda = (p,x) \in T^*_x M$ and $\|\cdot\|$ denotes the usual 
norm
on the cotangent bundle induced by the Riemannian metric 
on $M$. The manifold setting enables to include relevant physical systems such as rotating molecules~\cite{molecule,kochsugny}.

The first idea that comes to mind is that maybe both systems share some similar controllability properties. Let $L^2(M)$ be the space of square integrable functions with respect to the standard Riemannian volume. Since the Schr\"odinger equation preserves the $L^2$-norm, we denote by $\cS(M) = \{\psi \in L^2(M)\,:\,\|\psi\|_{L^2(M)} = 1\}$ the $L^2$-sphere and consider quantum and classical reachable sets in time less than $T\in \R_+$ from states $\psi_0 \in \cS(M)$ and $\lambda_0 \in T^*M$:
\begin{align*}
\cA^q_{\leq T}(\psi_0) &= \left \{\phi \in \cS(M) \, : \, \exists 0 \leq s \leq T, u \in \cU, \psi(\cdot) \text{ solution of (\ref{eq:schro}) } \text{ s.t. } \psi(0) = \psi_0,\; \psi(s) = \phi \right \},\\
\cA^c_{\leq T}(\lambda_0) &= \left \{\mu \in T^*M \, : \, \exists 0 \leq s \leq T, u \in \cU, \lambda(\cdot) \text{ solution of (\ref{eq:class}) }\text{ s.t. } \lambda(0) = \lambda_0,\; \lambda(s) = \mu \right \}.
\end{align*}
Here $\cU$ denotes the set of admissible controls. Under reasonable assumptions on $V$ and $W$, one can take $\cU$ to be the space of piecewise constant functions. In particular we will make the following general assumption:
\begin{ass}
\label{ass}
We assume that $V,W$ are $C^\infty$ functions such that the operators
$$
-\frac{\Delta}{2} + V(x) + uW(x)
$$
are essentially self-adjoint on $L^2(M)$ for every $u\in \R$. 
\end{ass}
Assumption~\ref{ass} holds, for instance, if $M$ is a compact manifold or if $M=\R^n$, $V$ is bounded from below and $W\in L^\infty(M)\cap C^\infty(M)$~\cite[Theorem X.28]{reed-simon}.

It is well-known that one cannot   generally expect full controllability of equation~\eqref{eq:schro} from a given state $\psi_0$ \cite{ball_and_marsden}, i.e., in general
\begin{equation}\label{eq:notcontrollable}
\bigcup_{T\geq 0} \cA^q_{\leq T}(\psi_0) \subsetneq \cS(M).
\end{equation}
Nevertheless, 
it has been proved that, under suitable
generic 
conditions, 
the left-hand side 
of \eqref{eq:notcontrollable} is dense in $\cS(M)$~\cite{generic}.
If  
\begin{equation*}
\overline{\bigcup_{T\geq 0} \cA^q_{\leq T}(\psi_0)} = \cS(M)
\end{equation*}
for every state $\psi_0 \in \cS(M)$, then we say that the Schr\"odinger equation~\eqref{eq:schro} is approximately controllable. The full and approximate classical controllabilities could be defined in the same way, but in view of the relation with quantum mechanics it makes sense to define approximate controllability of the system~\eqref{eq:class} as the property that there exists a dense set $\Upsilon\in T^* M$ such that for every initial state $\lambda_0 \in \Upsilon$
\begin{equation}
\label{eq:class_approx}
\overline{\bigcup_{T\geq 0} \cA^c_{\leq T}(\lambda_0)} = T^* M.
\end{equation}
For brevity we write ``QA" for ``quantum approximate" and ``CA" for ``classical approximate".

Coming back to our question, the first naive guess could be that CA controllability should imply QA controllability. However simple examples show that this is not the case in general. For example, it is known that the harmonic oscillator in a controlled electric field
\begin{equation*}
\ii\p_t \psi = -\frac{1}{2}\p_x^2 \psi + x^2\psi + u(t) x \psi, \quad \psi \in L^2(\R),
\end{equation*} 
is not QA controllable~\cite{oscil,coron_book}, even though the corresponding classical system
\begin{align*}
\dot{x} &= p, \\
\dot{p} &= -x - u,
\end{align*}
is controllable as it follows easily from the Kalman condition (for, instance, \cite{sontag}). On the other hand, if a classical system is not CA controllable, it does not necessarily follow that the corresponding quantum system is not approximately controllable either, since quantum effects such as tunneling allow a particle to move into classically forbidden regions. In the appendix we discuss an example of this kind. In the following we go more deeply inside the relation between classical and quantum controllability.

We define the controllability diameter of~\eqref{eq:schro} as
$$
T_q = \sup_{\psi_0 \in \cS(M)}\inf \left\{ T \geq 0 \, : \,\overline{\cA^q_{\leq T}(\psi_0)} = \cS(M) \right\},  
$$
i.e., $T_q$ is the smallest time such that any state in $\cS(M)$ can be transferred to an arbitrarily small neighborhood of any other state in time less than $T_q$. We define analogously the controllability diameter of the classical system~\eqref{eq:class} as
$$
T_c = \inf \left\{ T \geq 0 \, :\, \exists \Upsilon\subset T^*M \text{ dense s.t. }\forall \lambda_0 \in \Upsilon,\;  \overline{\cA^c_{\leq T}(\lambda_0)} = T^*M \right\}.
$$
The reason for 
considering, in the 
definition of the controllability diameter $T_c$, 
only a dense set of initial conditions is that we want to take into account situations in which a common  
critical point $\bar x$ for $V$ and $W$ prevents the 
usual 
approximate controllability property to hold true in $T^*M$, even if it 
holds in the punctured manifold $T^*M\setminus\{(\bar x,0)\}$. In this situation, we expect that such slightly reduced  controllability property of the classical system is reasonably informative about the controllability of the corresponding quantum system.

It is quite natural to expect that there is a connection between the two quantities $T_q$ and $T_c$. Indeed, semiclassical analysis allows to construct approximate solutions to the Schr\"odinger equation using the solutions of the corresponding Hamiltonian system. Notice however that if a system is not CA or QA controllable, then the corresponding controllability diameter is equal to infinity. Thus the controllability discussion above shows that we should not expect nice estimates such as $c T_c \leq T_q \leq C T_c$ for some positive constants $c,C$. Instead, something that is much more reasonable is the following implication that we give in the form of a conjecture.
\begin{conjecture}
\label{conjecture}
Let $M$ be a connected complete Riemannian manifold. For any 
Hamiltonian of the form~\eqref{eq:ham_class} 
the following implication is true
\begin{equation}
\label{eq:conjecture}
T_c > 0 \quad \Rightarrow \quad T_q > 0.
\end{equation} 
\end{conjecture}
This implication basically says that all the quantum effects which are due to non-locality of the Schr\"odinger equation have a limited effect at small times. Understanding if a system can or cannot be controlled in arbitrarily small time is important for applications, in particular when the system is subject to decoherence and hence it must be controlled in a time much smaller than the decoherence time~\cite{Schroedinger_cat}.

When $T_q = 0$ (respectively, $T_c = 0$), we say that the quantum system~\eqref{eq:schro} (respectively, the classical system~\eqref{eq:class}) is \emph{small-time controllable}. Most of the questions related to small-time controllability are completely open. In the classical case $T_c$ can indeed be equal to zero, since the classical harmonic oscillator in an electric field is a linear system, and all controllable linear systems have $T_c = 0$ (see, for instance, \cite{violent,sontag}). But in general the question is open even locally for polynomial systems~\cite{Agrachev1999}.

In the quantum case it is known that for systems evolving in a finite-dimensional Hilbert space (e.g., spin systems)
$$
\ii\dot{\psi} = (A+u(t)B)\psi,
$$
where $\psi \in \C^n$ and $A,B$ are Hermitian matrices, we necessarily have $T_q > 0$ for $n>1$. Indeed, if we denote by $e$ an eigenvector of $B$, one can check that the corresponding occupation 
probability cannot change too quickly, since it satisfies the  estimate
\begin{equation}
\label{eq:trans_prob}
|\langle \psi(t),e \rangle| \leq |\langle \psi(0),e \rangle| + \|A\|t.
\end{equation}
 (See~\cite{zero-time}).
Some upper and lower bounds on the controllability diameter in this case were studied in the articles~\cite{agrachev_thomas,brock}.

For a system of the form~\eqref{eq:schro} evolving in an infinite-dimensional space, $A$ is not a bounded operator anymore. Hence the estimate~\eqref{eq:trans_prob} is not an obstacle for small-time controllability. This gives hope that 
in some infinite-dimensional cases small-time controllability can be achieved. An example of such a system was given in~\cite{zero_time}, where the authors proved that
$$
\ii\p_t \psi = -|\p^2_x|^\alpha \psi + u(t)\cos x \psi,
$$ 
$$
\psi_0(-\pi) = \psi_0(\pi) = 0,
$$
is small-time controllable for $\alpha>5/2$. The  physically interesting case $\alpha = 1$ is however open.

In contrast to this positive result, negative results were presented in~\cite{coron_teismann,coron_teismann2018}, where the authors proved that for a particle in an electric field we always have $T_q > 0$. All the small-time uncontrollability results that we cite in the present paper are based on proving that some properties of the considered system 
are uniform with respect to the control. 
For example, in the finite-dimensional case we have estimates on the occupation probabilities~\eqref{eq:trans_prob} that satisfy this property. In~\cite{coron_teismann} the authors exploit instead the fact that, for an electric potential, Gaussian wave-packets remain localized along a given classical trajectory until a time (called the Ehrenfest time) that is independent of the control $u(t)$. Thus a Gaussian state has to remain almost Gaussian for small times and hence $T_q > 0$.

In this article we study some obstructions to small-time controllability that are related to the behaviour of projections of classical trajectories to the base manifold $M$. The paper is divided into two sections. In the first section we give some sufficient conditions to have $T_c =0$ or $T_c > 0$. In the second section we see how in specific examples one can deduce that $T_q > 0$ from $T_c > 0$. 

In particular, in Section~\ref{sec:class} we show along which directions in the classical phase space it is possible to move arbitrarily fast and 
we prove the following theorem.

\begin{theorem}
\label{thm:main_class}
Assume that $M = N_1 \times N_2$ is the product of two connected complete Riemannian manifolds $N_1,N_2$. Consider a Schr\"odinger equation of the form~\eqref{eq:schro} on $M$ and assume that the potentials $V, W: M\to \R$ satisfy Assumption~\ref{ass}. Suppose that there exists a nonempty open set $\Omega\in N_1$ such that 
\begin{enumerate}
\item the restrictions $W|_{\Omega\times \{y\}}$ are constant for all $y\in N_2$,
\item there exists a Lipschitz continuous function $c: \Omega\to \R$ such that 
$\|d_{1} V(x,y)\|_{T^*_{x} N_1} \leq c(x)$ 
for all $(x,y)\in M$,
\end{enumerate}
where $d_{1}V(x,y)$ denotes the differential at $x$ of $V(\cdot,y):N_1\to \R$. Under these assumptions $T_c > 0$. 
\end{theorem}

In Section~\ref{sec:quant} for the corresponding quantum case we prove the following theorem.
\begin{theorem}
\label{thm:main1}
Assume that conditions of Theorem~\ref{thm:main_class} are satisfied. Then $T_q > 0$. 
\end{theorem}

Let us discuss a couple of practical situations where Theorems~\ref{thm:main_class} and \ref{thm:main1} can be applied. Assume, for example, that $N_2$ is just a point. In this case $M = N_1$ and $W$ is constant on a nonempty open subset $\Omega$ of $M$ and $T_c > 0$ because in $\Omega$ the classical dynamics is completely independent of the control $u$. This means that until a classical particle exists $\Omega$ it is unaffected by the control and thus the positivity of the exit time from $\Omega$ is an obstruction for having $T_c = 0$. 

We note the fact that $T_q>0$ is not immediately obvious from the previous discussion, since the Schr\"odinger equation is not local. So even if initially the support of $\psi_0$ is concentrated in $\Omega$, a certain amount of the occupation probability 
exits  $\Omega$ instantaneously. This effect can be enough to ensure QA controllability of a system whose classical counterpart is not CA controllable as in the example presented in the Appendix. Nevertheless the rationale of Theorem~\ref{thm:main1} is that the tunneling effect is not strong enough to transfer occupation probability 
outside $\Omega$ in arbitrarily small time. Thus the state remains essentially localized for sufficiently small times.

Another situation when Theorem~\ref{thm:main1} applies is when $M = \R^n = \R^{n_1}\times \R^{n_2}$ endowed with a flat product metric. If we denote by $(x,y)$ the coordinates on $\R^{n_1}\times \R^{n_2}$, then the first condition of the theorem simply says that $W(x,y)= W(y)$ for all $x\in \Omega$. In this case the gradient lines of $W$ preserve the fibre bundle $\Omega \times \R^{n_2}$, and what one can show is that a classical particle again cannot leave $\Omega \times \R^{n_2}$ in arbitrarily small time, because the second condition ensures that the speed along the base $\Omega$ is uniformly bounded on each fibre $\{x\}\times \R^{n_2}$. Similarly to the previous case one can show that, if the wave function at the initial time is concentrated in $\Omega \times \R^{n_2}$,  it will stay essentially 
localized in this set for some positive time independent of the control $u$.

The proof of the theorem is based on the WKB method. We will discuss its geometric meaning, and we will see how one can extract information about the controllability diameter. After that we will prove the theorem above by a modification of the WKB ansatz.
A possible variant of the proof not relying on semiclassical estimates are suggested in 
Remark~\ref{thet}.

\section{Classical small-time controllability}
\label{sec:class}

\subsection{Small-time controllability in phase space}
 
We start by looking at which states can be reached in an arbitrarily small time from a given state $\lambda_0\in T^*M$. Let $\lambda_0 = (p_0,x_0)$ and as usual by $\pi: T^*M \to M$ we denote the projection to the base manifold $M$.

\begin{definition}
Let $\lambda_0,\lambda_1$ be in $T^*M$. 
We say that the Hamiltonian system \eqref{eq:class} \textit{can be steered in small time from $\lambda_0$ arbitrarily close to $\lambda_1$} if for any $T>0$ and any neighborhood $\Omega\subset T^* M$ of $\lambda_1$, there exists an admissible control function $u:[0,T']\to \R$ such that $T'\le T$ and the corresponding solution $\lambda(\cdot)$ of \eqref{eq:class} with initial condition $\lambda(0) = \lambda_0$ satisfies $\lambda(T') \in \Omega$.
\end{definition}

\begin{lemma}
\label{lem:useful}
Consider 
a Hamiltonian of the form~\eqref{eq:ham_class}. Given an initial state $\lambda_0$ one can steer \eqref{eq:class}
in small time from $\lambda_0$ arbitrarily close to 
\begin{enumerate}
\item the state $\lambda_0 + k d W(x_0)$, $\forall k\in \R$;
\item a state $\lambda$ whose projection $\pi(\lambda)$ is arbitrarily close to any chosen point on the geodesic issued from $x_0$ with initial covector $dW(x_0)$.
\end{enumerate}
\end{lemma}

\begin{proof}
Let us write
$$
H_g(\lambda) = \frac{\|p\|^2}{2}, \qquad H_V(\lambda) = V(x), \qquad H_W(\lambda) = W(x).
$$
Then the classical system~\eqref{eq:class} can be written as 
\begin{equation}
\label{eq:ham}
\dot{\lambda} = \vec{H}_g(\lambda) + \vec{H}_V(\lambda) + u(t)\vec{H}_W(\lambda)
\end{equation}
or, in local coordinates and using the Einstein summation convention, as
\begin{equation}
\label{eq:ham_coord}
\begin{array}{l}
\dot{x}^i = g^{ij}p_j,\\
\dot{p}_i = -\dfrac{\p g^{jk}}{\p x^i} p_jp_k - \dfrac{\p V}{\p x^i} - u(t) \dfrac{\p W}{\p x^i}.
\end{array}
\end{equation}
For every $\varepsilon>0$ we consider the constant control of the form $u(t) = -k/\varepsilon$ defined on the interval $[0,\varepsilon]$. We rescale the time variable by taking $t\to t/\varepsilon$. In this way we map the interval $[0,\varepsilon]$ to $[0,1]$ and the endpoint $\lambda(\varepsilon)$ is mapped to the endpoint $\lambda(1)$ of the Cauchy problem
\begin{equation}
\label{eq:Cauchy}
\dot{\lambda} = \varepsilon \vec{H}_g(\lambda)+\varepsilon \vec{H}_V(\lambda) - k\vec{H}_W(\lambda), \qquad \lambda(0) = \lambda_0.
\end{equation}
We obtain a system that depends smoothly on the parameter $\varepsilon$. Thus, by the continuous dependence of solutions of ODEs on parameters, the solution of~\eqref{eq:Cauchy} converges uniformly to the solution corresponding to $\varepsilon = 0$, which is exactly $t \mapsto \lambda_0 + tk d_{x_0} W$.  

The second part is proved similarly. We take $u(t) \equiv 0$ and for every $\varepsilon>0$ we consider the solution of~\eqref{eq:Cauchy} on the interval $[0,\varepsilon]$ with initial condition $\lambda(0)=\lambda_0  + (k/\varepsilon) d_{x_0} W$. It is well known that the joint rescaling $p \to p\varepsilon $, $t \to t/\varepsilon$ of the fibre variables and time is a symmetry of the geodesic equations. Indeed, from the coordinate expressions~\eqref{eq:ham_coord} it is easy to see that the solution of
$$
\dot{\lambda} = \vec{H}_g(\lambda) + \vec{H}_V(\lambda)
$$
with $\lambda(0) = \lambda_0  + (k/\varepsilon) d_{x_0} W$ on $[0,\varepsilon] $ is mapped to the solution of
$$
\dot{\lambda} = \vec{H}_g(\lambda) + \varepsilon^2\vec{H}_V(\lambda)
$$
on $[0,1]$ with $\lambda(0) = \varepsilon\lambda_0  + kd_{x_0} W$. Once again 
by continuous dependence of solutions of ODEs on the parameter,
we can conclude  that $\pi(\lambda(1))$ converges, as $\varepsilon\to 0$, to the endpoint at time $1$ of the geodesic issued from $x_0$ with initial covector $k dW(x_0)$.

Note that at this point we still need to justify why we can take $\lambda(0)=\lambda_0  + (k/\varepsilon) d_{x_0} W$ as the initial condition, because using the first statement we can only steer
\eqref{eq:class} from 
 $\lambda_0$ to a neighborhood $\Omega$ of $\lambda_0  + (k/\varepsilon) d_{x_0} W$. But since $\Omega$ can be made arbitrarily small, the result follows from the continuous dependence of solutions on the initial value.
\end{proof}

This result gives us a number of interesting corollaries.
\begin{corollary}
\label{cor:fast}
Given an initial state $\lambda_0$ such that $d W(\pi(\lambda_0)) \neq 0$, let $\gamma$ be the gradient curve of $\nabla W$ passing through $\pi(\lambda_0)$ and let $q_1$ be a point on $\gamma$. Then 
\eqref{eq:class}
can be steered in small time from $\lambda_0$  
to a state $\lambda_1$ such that $\pi(\lambda_1)$ is arbitrarily close to $q_1$.
\end{corollary}

\begin{proof}
We note that the geodesic starting from $x_0$ with initial covector $dW(x_0)$ has the same tangent line as the gradient curve $\gamma$ at $x_0$, as can be easily seen from~\eqref{eq:ham_coord}. Thus $\gamma$ can be approximated with geodesic arcs and we simply apply Lemma~\ref{lem:useful} several times using the continuous dependence of solutions of ODEs on initial values. 
\end{proof}

\begin{corollary}
\label{cor:cor}
Consider a classical system~\eqref{eq:class} with Hamiltonian
\begin{equation}
\label{eq:multi_ham}
H(t,\lambda) = \frac{\|p\|^2}{2} + V(x)+ \sum_{i=1}^{n} u_i(t) W_i(x),
\end{equation}
where $n = \dim M$, all potentials are $C^2$ and $u_i \in \R$. If the set
$$
O = \{x\in M \,:\, dW_1 \wedge \cdots \wedge dW_n(x) = 0 \}
$$
has empty interior, then~\eqref{eq:class} is small-time CA-controllable. If $\dim M = 1$ then the converse is true. Namely if~\eqref{eq:class} is small-time CA-controllable, then $O$ has an empty interior.
\end{corollary}

\begin{proof}
Even though the statement is about systems with several controls, the proof can be done by using Lemma~\ref{lem:useful} and Corollary~\ref{cor:fast}. Let us consider the set $F = \pi^{-1}(M\setminus O)\setminus M$ where the last $M$ should be thought as the zero section of $T^* M$. By construction $F$ does not contain fixed points of~\eqref{eq:class}, for which the dynamics is independent of the control. If $\lambda_0$ and $\lambda_1$ are two states in $F$, then $\pi(\lambda_0)$ can be connected to $\pi(\lambda_1)$ via a geodesic, whose lift to the cotangent bundle is given by $[0,1]\ni t \mapsto \mu(t)$. But since $\pi(\mu(0))=\pi(\lambda_0) \notin O$, 
we can decompose $\mu(0)$ in the basis of $dW_1,\dots,dW_n$, i.e., we can write 
$$
\mu(0) = \sum_{i=1}^{n} a_i dW_i(\pi(\lambda_0)) 
$$
for some constants $a_i \in \R$. We can then apply 
Corollary~\ref{cor:fast}
to the control system governed by the  Hamiltonian
$$
H_a(t,\lambda) = \frac{\|p\|^2}{2} + V(x)+ v(t)\sum_{i=1}^{n} a_i W_i(x),
$$
where $v(t)$ is the new control, i.e., we take $u_i(t) =a_i v(t)$ as controls in \eqref{eq:multi_ham}. It follows that we can steer \eqref{eq:class} in small time
from $\lambda_0$ arbitrarily close to some $\lambda$ with $\pi(\lambda)$ arbitrarily close to $\pi(\lambda_1)$.
In particular, we can assume that $\pi(\lambda)$ is in $O$ and,
using the same trick as above, we can choose $a_1,\dots,a_n$ such that
$$
\sum_{i=1}^{n} a_i dW_i(\pi(\lambda_1)) = \lambda_1 - \lambda.
$$ 
By applying the first statement of Lemma~\ref{lem:useful}, we conclude the proof
that \eqref{eq:class} is small-time CA-controllable.

In dimension one the emptiness of the interior of $O$ is also a necessary condition for small-time controllability, since on an open set where $W'(x) \equiv 0$ the Hamiltonian system~\eqref{eq:ham} is independent of the control and thus $T_c > 0$.
\end{proof}

Concerning  Conjecture~\ref{conjecture}, Corollary~\ref{cor:cor} shows that the set of one-dimensional systems with a nowhere vanishing potential $W$ is probably the best place to start looking for concrete examples of small-time QA controllable systems, since there are no other possible semi-classical obstructions than the vanishing of $W'$ on an open set. In the next subsection we show that even though being rare, there can be other obstructions in higher dimensions.

In order to find them, we define the \emph{controllability diameter in the configuration space} as
$$
T_c^{cs} = \inf \left\{ T \geq 0 \, :\, \exists \Upsilon\subset T^*M \text{ dense s.t. }\forall \lambda_0 \in \Upsilon,\;  \pi(\overline{\cA^c_{\leq T}(\lambda_0)}) = M \right\}.
$$
In this definition we essentially ignore information about the momenta and focus on the position of a particle in the configuration space. From the definition it clearly follows that
$$
T_c^{cs} \le T_c.
$$
And hence in order to prove Theorem~\ref{thm:main_class} we prove that $T_c^{cs}>0$.

\begin{proof}[Proof of Theorem~\ref{thm:main_class}]
By possibly taking a smaller subset in $\Omega$, we can assume that $T^*\Omega$ is a trivial bundle and we can introduce local coordinates $(x,p^x)$. Similarly by $(y,p^y)$ we denote a point in $T^*N_2$. Using the product structure of $M$, due to the assumptions of the theorem, in $\Omega\times N_2$ the Hamiltonian system~\eqref{eq:ham} can be written as 
\begin{equation}
\label{eq:ham_comparison}
\begin{array}{l}
\dot{x}^i = g^{ij}_x p^x_j,\\
\dot{p}^x_i = -\dfrac{\p g_x^{jk}}{\p x^i} p^x_jp^x_k - \dfrac{\p V(x,y)}{\p x^i} ,\\
\dot{y}^i = g^{ij}_y p^y_j,\\
\dot{p}^y_i = -\dfrac{\p g_y^{jk}}{\p y^i} p^y_jp^y_k - \dfrac{\p V(x,y)}{\p y^i}-u(t)\dfrac{\p W(x,y)}{\p y^i}.
\end{array}
\end{equation}
We note that 
$$
-K(x)c(x)\leq  -\left| \dfrac{\p V}{\p x^i}(x,y)\right|\leq - \dfrac{\p V(x,y)}{\p x^i} \leq  \left| \dfrac{\p V}{\p x^i}(x,y)\right| \leq  K(x)c(x),
$$
where $K:\Omega \to [0,+\infty)$ is a locally Lipschitz function such that $|\p_{x_i}V(x,y)| \leq K(x)\|d_1V(x,y)\|_{T_{x}N_1}$ for every $(x,y)\in \Omega\times N_2$. Its existence follows from the equivalence of norms in finite-dimensional spaces. 

The result now follows from Chaplygin's lemma that states that if $z(\cdot)$ and $\tilde z(\cdot)$ are solutions of the differential equations
$$
\dot{z}=f(z), \qquad \dot{\tilde{z}}=\tilde f(\tilde{z}), \qquad z,\tilde{z}\in \R^n
$$
with $f,\tilde{f}$ locally Lipschitz, $f(z)\leq \tilde{f}(z)$ coordinate-wise for all $z\in \R^n$, and $z(0) \leq \tilde{z}(0)$, then $z(t) \leq \tilde{z}(t)$ for all $t\geq 0$ for which both solutions exist. 

In our case we apply the Chaplgin lemma to 
compare the solutions of ~\eqref{eq:ham_comparison} with those of
the two control systems 
\begin{equation}
\label{eq:new}
\begin{array}{l}
\dot{x}^i = g^{ij}_x p^x_j,\\
\dot{p}^x_i = -\dfrac{\p g_x^{jk}}{\p x^i} p^x_jp^x_k \pm K(x)c(x),\\
\dot{y}^i = g^{ij}_y p^y_j,\\
\dot{p}^y_i = -\dfrac{\p g_y^{jk}}{\p y^i} p^y_jp^y_k - \dfrac{\p V(x,y)}{\p y^i}-u(t)\dfrac{\p W(x,y)}{\p y^i},
\end{array}
\end{equation}
one for each choice of the sign in front of the term $K(x)c(x)$.
The right-hand side of the equations for the variables $(x,p^x)$ in 
\eqref{eq:new}
is now autonomous, locally Lipschitz, and independent of the control. Thus for a given initial state $\lambda_0\in T^* \Omega\times T^*N_2$ the corresponding trajectory will remain in $T^* \Omega\times T^*N_2$ for small times independent of the control function $u(\cdot)$. So we can apply the Chaplygin lemma for each interval on which $u(\cdot)$ is constant, and the projection 
onto $N_1$
of the solution of the Hamiltonian system~\eqref{eq:ham_comparison} with the same initial condition $\lambda_0$ will also remain in $\Omega$ for all times
in a small interval independent of the control function $u(\cdot)$. Moreover, the bound on the exit time from $T^* \Omega\times T^*N_2$ can be made uniform with respect to all initial conditions in a neighborhood of $\lambda_0$. Hence $T_c^{cs}>0$ and $T_c > 0$.
\end{proof}

\section{Obstructions to small-time controllability of quantum systems}
\label{sec:quant}

\subsection{WKB method}\label{sec:WKB}

The goal of this section is to prove 
Theorem~\ref{thm:main1}. We start by recalling the WKB method which is going to be our main tool.
 Let us recall quickly some basic definitions that we need in the geometric presentation of the WKB method (for more details see~\cite{anantharaman} or~\cite{geometric_quant}). Let $M$ be a Riemannian manifold and
let us denote  by $\langle \cdot,\cdot \rangle$ the scalar product on $TM$. 
Recall that the gradient of a function $f$ is characterized by the identity 
$$
\langle \nabla f, Y \rangle = Y[f], \qquad \forall Y \in \Vect(M),
$$
Using a local orthonormal frame $X_1,\dots,X_n$ of vector fields we can equivalently write
$$
\nabla f = \sum_{i=1}^n X_i[f]X_i.
$$
The volume $\Vol$ allows to define the divergence of a vector field $X \in \Vect(M)$ as 
$$
\dive(X) \Vol = L_X \Vol,
$$ 
where $L_X$ is the Lie derivative in the direction $X$. Thus we can define the Laplace operator in the usual way as $\Delta f = \dive \nabla f$. In a local orthonormal frame the operator $\Delta$ has the form
$$
\Delta = \sum_{i=1}^n X_i^2 + \dive (X_i) X_i.
$$
From here it is easy to verify that the formula for the Laplacian of a product of two functions $a,b$ is the same as in the Euclidean case: 
$$
\Delta (ab) = a\Delta b + 2\langle \nabla a, \nabla b\rangle  +  b\Delta a.
$$

Let us now consider a Schr\"odinger equation on $M$ of the form
\begin{equation}
\label{eq:sch1}
\ii\hbar \p_t \psi  +\frac{\hbar^2\Delta \psi}{2} - V(t,x)\psi = 0,
\end{equation}
where $V$ is a smooth time-dependent potential. 
Here we introduce the parameter $\hbar$ since we are going to study 
the formal expansion of \eqref{eq:sch1} with respect to $\hbar$. 
We start with the usual WKB ansatz
$$
\tilde{\psi}(t,x) = a(t,x)e^{\frac{\ii}{\hbar}S(t,x)}.
$$
By plugging this ansatz into~\eqref{eq:sch1}, we obtain the expression
$$
\ii\hbar \left(\p_t a + \frac{\ii}{\hbar}a \p_t S\right) + \frac{\hbar^2}{2}\left( \Delta a + \frac{2\ii}{\hbar}\langle \nabla a,\nabla S\rangle + \frac{\ii}{\hbar}a\Delta S - \frac{1}{\hbar^2}a\|\nabla S\|^2\right) - V(t,x)a = 0.
$$
We now collect terms of different order in $\hbar $. For $\hbar^0$ we obtain the Hamilton--Jacobi equation
\begin{equation}
\label{eq:HJ}
\p_t S + \frac{\|\nabla S\|^2}{2} + V(t,x)= 0
\end{equation}
and for the order $\hbar^1$ the transport equation
\begin{equation}
\label{eq:trans}
\p_t a + \langle \nabla a,\nabla S \rangle + \frac{\Delta S}{2}a = 0.  
\end{equation}
If we can solve the last two equations, then we are also able to construct a WKB approximation $\tilde{\psi}$, which satisfies the following 
Schr\"odinger equation
\begin{equation}
\label{eq:sch3}
\ii\hbar \p_t \tilde \psi  +\frac{\hbar^2\Delta \tilde \psi}{2} - V(t,x)\tilde \psi = \hbar^2\frac{\Delta a}{2}.
\end{equation}
Thus if $\psi(t,x)$ is a solution of~\eqref{eq:sch1} with initial condition $\psi(0,x) = \tilde{\psi}(0,x)$, by Duhamel's formula we have
\begin{equation}
\label{eq:error}
\psi(t,\cdot) - \tilde{\psi}(t,\cdot) = \frac{\hbar^2}{2}\int_0^t (U(s,t)\Delta a)(s,\cdot)ds,
\end{equation}
where $U(s,t)$ is the propagator from time $s$ to time $t$ of~\eqref{eq:sch1}.

Let us now comment on the solutions of~\eqref{eq:HJ} and~\eqref{eq:trans}. We begin with the Hamilton--Jacobi equation~\eqref{eq:HJ}. With each smooth function $S_0:M\to \R$ we can associate a differential 1-form $dS_0$, which can be identified with its graph $L = \{(x,d_x S_0)\,:\,x\in M\}$. 
One can easily check that this graph is a Lagrangian submanifold, i.e., it has dimension $\dim M$ and the symplectic form of $T^* M$ vanishes on $L$. 

In order to find a solution of the Hamilton--Jacobi equation~\eqref{eq:HJ} we can try to construct it from a curve of Lagrangian submanifolds $L_t$. We denote by $\Phi^t:T^*M\to T^*M$ the flow from time $0$ to time $t$ of the classical Hamiltonian system associated with the Hamiltonian
\begin{equation}
\label{eq:hamiltonian}
H(\tau,p,x)= \frac{\|p\|^2}{2} + V(\tau,x).
\end{equation}
Assuming that $L_0 =\{(x,d_x S_0)\,:\,x\in M\}$, where $S_0:M\to\R$ is smooth,
we define $L_t = \Phi^t(L_0)$. If there exists a smooth function $S:[0,\varepsilon]\times M\to \R$ such that $S(0,\cdot)=S_0$ and
each $L_t$, $t\in [0,\varepsilon]$, is the graph of the differential of $S(t,\cdot)$, 
then $S$ is a solution of~\eqref{eq:HJ} as follows from the method of characteristics~\cite{anantharaman}. Of course this is not possible, in general, for all $t$ larger than $\varepsilon$, because the restriction $\pi|_{L_t}$ can stop being one-to-one. We will explain how to avoid this issue in the next section.

Let us for now assume that we were able to construct a solution of the Hamilton--Jacobi equation~\eqref{eq:HJ}. Then we can solve the transport equation~\eqref{eq:trans} as well, via the method of characteristics. 
Since 
$\langle \nabla a, \nabla S \rangle = \nabla S[a]$,
the equation for characteristics can be written as 
\begin{align}
\dot{t} &= 1\label{eq:primio},\\
\dot{x} &= \nabla S,\label{eq:seco}\\
\dot{a} &= -\frac{\Delta S}{2}a\label{eq:terzio}.
\end{align}
The first equation simply says that we can parametrize characteristic curves using the original time. In the second one we note that by construction $\nabla S = \langle d_{x} S,\cdot \rangle$, which is the projection of the Hamiltonian vector field $\vec H$ to $M$. Therefore solutions of \eqref{eq:seco} 
with initial value $x(0) = x_0$ are given by $ x(t,x_0) = \pi \Phi^t(x_0,d_{x_0} S_0)$. Let us denote the flow of \eqref{eq:seco} from time $0$ to time $t$ 
by $G^t$. Then $G^t$ 
takes a point $x_0$ in $M$, lifts it to the cotangent bundle as $(x_0,d_{x_0} S_0)$, transfers it along the Hamiltonian flow $\Phi^t$, and projects it back to $M$. Then the characteristics method gives us a solution of \eqref{eq:terzio} of the form
\begin{equation}
\label{eq:a}
a(x,t) = a_0((G^t)^{-1} x) \exp \left(-\frac{1}{2}\int_0^t \Delta S(\tau, x(\tau))d\tau \right).
\end{equation}

We can further simplify the solution by giving a geometric interpretation of the exponential term. We claim that 
$$
J_t(x) = \exp \left(\int_0^t \Delta S(\tau, G^\tau x)d\tau\right)
$$
is equal to the Jacobian of the flow $G^t$. 
Here the Jacobian is 
defined intrinsically via the identity
$$
 (\det d G^t) \Vol =(G^t)^* \Vol.
$$
Differentiating this equality with respect to $t$, we obtain that
$$
\left(\frac{d}{dt}\det d G^t\right) \Vol = \frac{d}{dt} (G^t)^* \Vol = L_{\nabla S} \Vol = (\Delta S) \Vol.
$$
Solving this differential equation, we obtain the equality
$$
J_t(x) = \det d G^t (x).
$$

\subsection{Relation to the standard calculus of variations}

From formulas~\eqref{eq:error} and~\eqref{eq:a} it follows that the $L^2$ norm of the WKB approximation blows up if $J_t(x)$ is zero for some $t$ and $x$. The set of zeros of $J_t$ is called the \emph{caustic} and, if it is nonempty, the WKB approximation breaks down. If we multiply the WKB ansatz by a smooth function whose support has empty intersection with the caustic for sufficiently small times, then this new approximation is going to be a well-defined smooth function.
 The goal of this section is to recall some results that will help us to prove that caustics in certain regions of the configuration space cannot develop for sufficiently small times.

In order to prove it, we use a link between the caustics and the minimality of extremal curves of the action functional
$$
\cS[x(\cdot)] =  S_0(x_0) + \int_0^t \left(\frac{\| \dot{x}(s)\|^2}{2} - V(s,x(s))\right)ds.
$$
It is known that each extremal curve must satisfy a Hamiltonian system with Hamiltonian given by~\eqref{eq:hamiltonian} and with initial momentum $p(0) = d_{x(0)} S_0$~\cite{calc_var}. Thus, 
the wavefront 
at time $t$, given by the endpoints at time $t$ of all such extremal curves in $T^*M$, coincides with the Lagrangian surface $L_t$ from which we constructed solutions of the Hamilton--Jacobi equation in the previous section.

Caustics play an important role in 
calculus of variations: in the case of the functional $\cS$ they indicate whether or not there are infinitesimal variations of the extremal curve that decrease the value of $\cS$ with fixed final point~\cite{calc_var}. An intersection point between a given extremal trajectory and the caustic is called a \emph{conjugate point} and the corresponding time a \emph{conjugate time}. Techniques from calculus of variations allow to estimate the time $t$ at which 
$L_t$ starts developing caustics. In particular, we will rely on the following result.
\begin{theorem}[\!\!\cite{calc_var}]\label{thm:cited1}
Sufficiently small arcs of extremal curves of $\cS$ do not contain conjugate points.
\end{theorem}
This is a consequence of the Legendre--Clebsch condition, which is verified for the functional $\cS$. It is a sufficient condition for local minimality of small arcs of extremal curves. 
\begin{corollary}
The first conjugate time is a lower semi-continuous function from $M$ to $(0,+\infty]$ of the initial point $x \in M$. In particular, the restriction of the  first conjugate time to a compact subset of 
$M$ has a 
positive lower bound.
\end{corollary}
The positivity of the first conjugate time follows from Theorem~\ref{thm:cited1}. Its lower-semicontinuity follows from the definition of the caustic as the set of zeroes of the smooth function $\det dG^t(x)$.

In terms of the functions $a$ and $J_t$ introduced in the previous section, we can rephrase the second part of the previous corollary as follows.
\begin{corollary}
\label{cor:opt}
Let $\Omega$ be compactly contained in $M$. 
The function $|J_t|$ restricted to 
$\Omega$ 
has a uniform positive lower bound 
for sufficiently small times.
In particular, given $a_0:\Omega\to\R$ smooth and $\Omega_0$ compactly contained in $\Omega$, there exists 
 $\varepsilon>0$ such that 
the 
system of equations~\eqref{eq:seco}--\eqref{eq:terzio} admits a smooth solution $(S,a)$ on $[0,\varepsilon]\times \Omega_0$ that depends only on $a_0$, $S_0|_{\Omega}$, and $V|_{[0,\varepsilon]\times \Omega}$. 
\end{corollary}

\subsection{Proof of Theorem~\ref{thm:main1} in the case when $N_2$ is a point}

In the case when $N_2$ is just a point, we have that $M = N_1$, the second condition of Theorem~\ref{thm:main1} is trivially satisfied and the first condition is simply given by $W|_{\Omega} = c$.

Let us fix a nonempty subset $\Omega'$ compactly contained in $\Omega$. 
In the open set $\Omega$ the classical dynamics are independent of the control, 
so it will take some time for each trajectory starting from $\Omega'$ to exit from $\Omega$, uniformly with respect to the control function. 
We wish to exploit this fact to prove that not only $T_c$, but also $T_q$ is positive.

Consider 
a smooth cut-off function $\chi$ 
whose support is equal to $\overline{\Omega'}$.
Fix two smooth functions $a_0,S_0:\Omega\to \R$
such that 
\begin{equation}\label{eq:normalizz}
\|\chi a_0\|_{L^2(\Omega)}=1.
\end{equation}
 Following 
 Corollary~\ref{cor:opt}, 
 there exists a time $\varepsilon>0$ such that, for every 
admissible control $u:[0,\varepsilon]\to \R$,
the 
system of equations~\eqref{eq:seco}--\eqref{eq:terzio} admits a smooth solution $(S,a)$ on $[0,\varepsilon]\times \Omega'
$ depending only on $a_0$, $S_0$, and $V|_{\Omega}$.

Then, for every 
admissible control $u:[0,\varepsilon]\to \R$,
define
\begin{equation}\label{eq:ansatz0}
\varphi(t,x) = \chi(x)\tilde \psi(t,x) \exp\left(  - \ii c\int_0^t u(s)ds\right),\qquad t\in [0,\varepsilon],\;x\in M,
\end{equation}
where $\tilde{\psi}(t,x)=a(t,x)e^{\ii S(t,x)}$ (with the 
convention that $\chi(x)\tilde \psi(t,x)=0$ for $x\notin \supp\chi$, even if $\tilde \psi(t,x)$ is not defined).
Notice that $\tilde{\psi}(0,\cdot)$ is in $\cS(M)$ thanks to \eqref{eq:normalizz}.
If we plug the ansatz above in the Schr\"odinger equation we find
\begin{align*}
\exp\left(   \ii c\int_0^t u(s)ds\right) &\left( \ii\p_t + \frac{\Delta}{2} - V - u(t) W\right) \varphi =\left( \ii\p_t + \frac{\Delta}{2} - V\right) (\chi \tilde \psi) = \\
=& \chi \left( \ii\p_t + \frac{\Delta}{2} - V\right) \tilde \psi + \langle \nabla \chi, \nabla \tilde{\psi} \rangle + \frac{\Delta \chi }{2}\tilde \psi  
= 
\chi \frac{\Delta a}{2} +  \langle \nabla \chi, \nabla \tilde{\psi} \rangle + \frac{\Delta \chi }{2}\tilde \psi,\end{align*}
where the first equality is a consequence of the identity  $\chi(W - c) \equiv 0$ and the last one follows from \eqref{eq:sch3}. 

Define $r:[0,\varepsilon]\times M\to \C$ by
\[
r(t,x)=\exp\left(-   \ii c\int_0^t u(s)ds\right)\left(\chi(x) \frac{\Delta a(t,x)}{2} +  \langle \nabla \chi(x), \nabla \tilde{\psi}(t,x) \rangle + \frac{\Delta \chi(x) }{2}\tilde \psi(t,x)\right),
\]
so that 
\[\left( \ii\p_t + \frac{\Delta}{2} - V - u(t) W\right) \varphi =r\qquad \mbox{on}\ [0,\varepsilon]\times M.\]
Notice that $r$ is smooth with respect to $x$ and that it depends only on $a_0$, $S_0$ and $V|_{\Omega}$, and not on the control function $u$.

By the Duhamel formula we have
$$
\varphi(t,\cdot) - \psi(t,\cdot) = \int_0^t U(s,t) r(s,\cdot) ds,
$$
where $U(s,t)$ is the propagator from time $s$ to time $t$ of~\eqref{eq:schro}
and 
$\psi$ is the solution of (\ref{eq:schro}) 
with initial condition 
$$
\psi(0,x) = \varphi(0,x),
$$
that is, $\psi(t,\cdot)=U(0,t)\varphi(0,\cdot)$.
Using the fact that $U(s,t)$ is unitary we obtain 
\begin{equation}\label{eq:restt}
\|\varphi(t,\cdot) - \psi(t,\cdot)\|_{L^2(M)} \leq \int_0^t \|U(s,t) r(s,\cdot)\|_{L^2(M)} ds \leq \int_0^t \| r(s,\cdot)\|_{L^2(M)}ds.
\end{equation}
Up to eventually reducing 
$\varepsilon$, assume that
\begin{equation}\label{eq:std}
\delta:=\int_0^\varepsilon \| r(s,\cdot)\|_{L^2(M)}ds
<1.
\end{equation}
Take now $\psi_1\in \cS(M)$ supported in $M\setminus \Omega$, so that $\|\psi_1-\varphi(t,\cdot)\|_{L^2(M)}=1+\|\varphi(t,\cdot)\|_{L^2(M)}\ge 1$ for every $t\in [0,\varepsilon]$ and every control law $u$.
Then, by triangular inequality and because of \eqref{eq:restt} and \eqref{eq:std}, for every $t\in [0,\varepsilon]$ and every $u(\cdot)$,
\[\|\psi_1-\psi(t,\cdot)\|_{L^2(M)}\ge \|\psi_1-\varphi(t,\cdot)\|_{L^2(M)}-\|\varphi(t,\cdot) - \psi(t,\cdot)\|_{L^2(M)}\ge 1-\delta>0,\]
concluding the proof that $T_q>0$.

\subsection{Proof of Theorem~\ref{thm:main1}}

In order to complete the proof of Theorem~\ref{thm:main1}, we wish to exploit the product structure of the manifold $M$. The Schr\"odinger equation takes the form
\begin{equation}
\label{eq:double_schro}
\ii\p_t \psi = -\frac{\Delta_1 \psi  + \Delta_2 \psi}{2} + V\psi + u W\psi,
\end{equation}
where $\Delta_1,\Delta_2$ are the Laplace--Beltrami operators on $N_1$, $N_2$.
 Similarly to the case when $N_2$ is a point, we 
 fix $\Omega'$ compactly contained in $\Omega$ and we
 look for an approximate solution of the form
\begin{equation}
\label{eq:ansatz}
\varphi(t,x,y)=\chi(x)\psi_1(t,x)\psi_2(t,y)
\end{equation}
where $x\in N_1$, $y\in N_2$, and 
$\chi$ is a smooth cut-off function whose support is $\overline{\Omega'}$. We can assume without any loss of generality that $\Omega$ is a %normal 
coordinate neighborhood centered at some $x_0\in \Omega'$. 

The first condition of the theorem tells us that $W(x,y) = \tilde W(y)$ for some function $\tilde W \in C^2(N_2)$ and all $x \in \Omega$, $y\in N_2$. So let us assume that $\psi_2
$ is a solution of 
$$
\ii\p_t \psi_2(t,y) = -\frac{\Delta_2 \psi_2(t,y)}{2} + V(0,y)\psi_2(t,y)+u(t)\tilde W(y)\psi_2(t,y).
$$
Plugging the ansatz~\eqref{eq:ansatz} into the Schr\"odinger equation~\eqref{eq:double_schro} above we find that
\begin{align*}
&\left( \ii\p_t  + \frac{\Delta }{2} - V(x,y) - u(t)W(x,y)\right)\varphi(t,x,y) = \psi_2(t,y)\left(\chi(x)\left(\ii\p_t  +\frac{\Delta_1 }{2}\right)\psi_{1}(t,x) + \right.\\
+ &\langle \nabla \chi(x), \nabla \psi_1(t,x) \rangle + \frac{\Delta_1 \chi(x)}{2}\psi_1(t,x) + \chi(x)(V(0,y)-V(x,y))\psi_1(t,x)+\\
+& \left.u(t)\chi(x)(W(x,y) - \tilde{W}(y))\psi_1(t,x)\right) =: r(t,x,y).
\end{align*}
Notice that, by construction,
$$
\chi(x)(W(x,y) - \tilde{W}(y)) \equiv 0.
$$ 

As in the previous section, 
let $a_0,S_0:\Omega\to \R$ be smooth and, thanks to Corollary~\ref{cor:opt},
 fix $\varepsilon>0$ such that
 there exist $a,S:[0,\varepsilon]\times \Omega'\to \R$ smooth 
such that
$$
\psi_1(t,x) = a(t,x)e^{\ii S(t,x)}
$$ 
is a solution of
$$
\ii\p_t \psi_1 + \frac{\Delta_1 \psi_1}{2} = \frac{\Delta_1 a}2
$$
on $[0,\varepsilon]\times\Omega'$.

Then 
\begin{align*}
r(t,x,y)=\psi_2(t,y)\Big(\chi(x) \frac{\Delta_1 a}2
+ \langle \nabla \chi(x), \nabla \psi_1(t,x) \rangle + \frac{\Delta_1 \chi(x)}{2}\psi_1(t,x)\\
+ \chi(x)(V(0,y)-V(x,y))\psi_1(t,x)\Big).
\end{align*}

The $L^2$ norm of  
\[M\ni (x,y)\mapsto \psi_2(t,y)\Big(\chi(x) \frac{\Delta_1 a}2
+ \langle \nabla \chi(x), \nabla \psi_1(t,x) \rangle + \frac{\Delta_1 \chi(x)}{2}\psi_1(t,x)\Big)\]
is uniformly bounded with respect to $t\in [0,\varepsilon]$.

Let us now focus on the term of $r$ involving $V$. 
In local coordinates on $\Omega$ we have 
$$
\chi(x)(V(0,y)-V(x,y))= \chi(x) \int_0^1   \langle \nabla_1 V(tx,y),x  \rangle_E dt \leq K(x) \chi(x)\|x\|_E \int_0^1 c(tx)dt
$$
where $\langle\cdot,\cdot\rangle_E$, $\|\cdot\|_E$ are the Euclidean scalar product and norm
 and $K(x)$ is a function that comes from equivalence of the Euclidean norm and the norm coming from the Riemannian scalar product. 
 We then deduce that the $L^2$ norm of 
\[M\ni (x,y)\mapsto \psi_2(t,y)\chi(x)(V(0,y)-V(x,y))\psi_1(t,x)\]
 is uniformly bounded with respect to $t\in [0,\varepsilon]$.

Summarizing everything we find that,  as in the previous section, $\varphi$ is an approximate solution of \eqref{eq:double_schro} concentrated in $\Omega \times N_2$ with a uniformly bounded error term
$$
\int_0^t \|r(s,\cdot)\|_{L^2(M)} ds
$$
which tends to $0$ as $t\to 0+$ uniformly with respect to $u$. Thus the occupation probability cannot be transferred outside $\Omega \times N_2$ in arbitrarily small time.
\begin{remark}\label{thet}
An alternative construction of the ansatz $\varphi$ in \eqref{eq:ansatz0} and \eqref{eq:ansatz}
could be obtained by replacing $\tilde \psi$ and $\psi_1$ by local regular (bounded in $H^1$) solutions of the Schr\"odinger equation. 
We have chosen to rely on the WKB approximation, since it permits a better 
%understanding of the 
parallel with the controllability of the corresponding classical system. 
\end{remark}

\begin{akn}
The authors would like to thank Patrick Gerard, Camille Laurent and Holger Teismann for helpful discussions and useful remarks.
\end{akn}

\section*{Appendix: An example of a quantum, but not classically controllable system}
\label{app}

In~\cite{quantum_control} it was proven that the harmonic oscillator with Gaussian control potential
\begin{equation}
\label{eq:gauss}
 \ii\p_t \psi(t,x) = \left( -\p_x^2+ x^2 + u(t) e^{ax^2 + bx +c} \right)  \psi (t,x)
\end{equation}
is controllable for a generic choice of real numbers $a<0,b,c$. In particular, the Gaussian must be non-centered in order to avoid reflection symmetry. It was an application of the following theorem proven in the same article.

\begin{theorem}
\label{thm:control_quant}
Consider a controlled evolution %abstract Schr\"odinger 
equation on a Hilbert space ${\cal H}$ of the form
\begin{equation}
\label{eq:schro_abstr}
\ii\p_t \psi = \left( A + v(t) B \right)  \psi,
\end{equation}
where 
$v(t)$ takes values in an interval $(0,\delta)$ and $A,B$ are self-adjoint operators on ${\cal H}$.
%and  for some $\delta > 0$. 
Suppose that $A$ has discrete spectrum $\{\lambda_i\}_{i\in \N}$, that the 
corresponding eigenvectors $\{\phi_i\}_{i\in \N}$ form an orthonormal basis of ${\cal H}$ and belong to $D(B)$. Suppose, moreover,  that
$A+v B:\mathrm{span}\{\phi_i\mid i\in \N\}\to {\cal H}$ is essentially self-adjoint  for every $v\in (0,\delta)$. 
  Denote by $b_{ij}=\langle B\phi_i,\phi_{j} \rangle$ the components of $B$ in this basis and by $B^{(k)}$ the first principal minor of order $k$. If 
\begin{itemize} 
\item the elements of $\{\lambda_{i+1}-\lambda_i\}_{i\in \N}$ are $\Q$-linearly independent,
\item for any $j \in \N$ there exists $k \geq j$, such that $B^{(k)}$ is connected,
\end{itemize}
then the Schr\"odinger equation~\eqref{eq:gauss} is approximately controllable.
\end{theorem}

Recall that with each Hermitian matrix $C$ we can associate a graph with $n$ vertices in such a way that the vertices $i$ and $j$ are connected by an edge if and only if $c_{ij}\neq 0$. Then the matrix $C$ is said to be \emph{connected} if the corresponding graph is connected.

Note that Theorem~\ref{thm:control_quant} does not apply immediately to the harmonic oscillator since $\lambda_{i+1}-\lambda_i = 1$ for all $i\in \N$. However, it is proven in~\cite{quantum_control} that for a suitable choice of constants $a,b,c$ and almost every $\mu \in \R$ the operator 
$$
A_\mu = -\p_x^2+ x^2 + \mu e^{ax^2 + bx +c}
$$
 has $\Q$-linearly independent differences $\lambda_{i+1}-\lambda_i$. This result is proven using some analyticity arguments, which reduce the proof of the statement to the proof of $\Q$-linear independence of the elements $b_{ii}$. Thus one can prove controllability of the harmonic oscillator with a Gaussian potential by considering an equivalent system of the form~\eqref{eq:schro_abstr} with
$$
A = A_\mu, \quad B =  e^{ax^2 + bx +c}, \quad v(t) = u(t) - \mu.
$$

Let us fix $a,b,c$ such that system~\eqref{eq:gauss} is controllable. We claim that there exists $\varepsilon>0$ such that 
$$
\ii\p_t \psi(t,x) = \left( -\p_x^2+ x^2 + u(t) \chi_{\{x > \varepsilon\}}e^{ax^2 + bx +c} \right)  \psi (t,x)
$$
is also approximately controllable. Indeed, let us denote $\hat B = \chi_{\{x > \varepsilon\}}e^{ax^2 + bx +c}$. Then using the arguments of~\cite{quantum_control} it is sufficient to prove that $\hat b_{ii}$ are all $\Q$-linearly independent and that $\hat B^{(k)}$ is connected for every $k$.

We know from the explicit formula for eigenfunctions of the harmonic oscillator that 
$$
\hat b_{ij}
 = b_{ij} - C_{ij}\int_{-\varepsilon}^{\varepsilon} e^{(a-1)x^2 + bx + c} H_i(x) H_j(x) dx,
$$
where $C_{ij}$ are constants and $H_i$ are Hermite polynomials. We note that under the integral we have analytic functions, since $H_i$ are polynomials. Hence the integrals 
$$
f_{ij}(\varepsilon) =  C_{ij}\int_{-\varepsilon}^{\varepsilon} e^{(a-1)x^2 + bx + c} H_i(x) H_j(x) dx
$$
depend analytically on $\varepsilon$ and converge to $0$ as $\varepsilon\to 0$. 
Hence, if 
$b_{ij}$ is different from $0$ then  
the set $f_{ij}^{-1}(b_{ij})$ has zero measure, which means that 
 $\hat b_{ij} \neq 0$
for almost every $\varepsilon > 0$. 
Therefore, since the finite union of zero measure sets has measure zero, 
if  the matrix $B^{(k)}$ is connected then 
 $\hat B^{(k)}$
is also connected for almost every $\varepsilon > 0$. Finally, for almost every $\varepsilon>0$ the following holds true: for any $j \in \N$ there exists $k \geq j$, such that $\hat B^{(k)}$ is connected.
 
Similarly we can prove that the set of $\varepsilon$ for which the elements of $\{\hat{b}_{ii}\}_{i\in \N}$ are $\Q$-linearly dependent has measure zero. Indeed if  the elements of $\{\hat{b}_{ii}\}_{i\in \N}$ are linearly dependent, then there exists a finite number of $\lambda_1,\dots,\lambda_k \in \Q$, such that 
$$
0 = \sum_{i=1}^k \lambda_i \hat{b}_{ii} = \sum_{i=1}^k \lambda_i b_{ii} - \sum_{i=1}^k \lambda_i f_{ij}(\varepsilon).
$$
Reasoning as above, 
$$
\left(\sum_{i=1}^k \lambda_i f_{ij}\right)^{-1}\left( \sum_{i=1}^k \lambda_i b_{ii} \right)
$$
has measure zero for each fixed set $\lambda_1,\dots,\lambda_k \in \Q$. Thus from countability of $\Q$ the claim follows.

On the other hand, 
the classical system of the form~\eqref{eq:class} with Hamiltonian
$$
H(t,p,x) = p^2 + x^2 + u(t) \chi_{\{x > \varepsilon\}}e^{ax^2 + bx +c}
$$
is not approximately controllable, because discs centered at the origin in phase space and of radius smaller than $\varepsilon$ are invariant under the Hamiltonian flow.

\bibliographystyle{siam}
\bibliography{references}

\end{document}